\documentclass[12pt,english]{elsarticle}
\usepackage{bm}
\usepackage{amsmath}
\usepackage{amsthm}
\usepackage{amssymb}

\makeatletter
\numberwithin{equation}{section}
\numberwithin{figure}{section}
\theoremstyle{plain}
\newtheorem{thm}{\protect\theoremname}[section]
\theoremstyle{plain}
\newtheorem{lem}[thm]{\protect\lemmaname}
\theoremstyle{definition}
\newtheorem{example}[thm]{\protect\examplename}
\theoremstyle{definition}
\newtheorem{assumption}[thm]{\protect\assumptionname}
\theoremstyle{definition}
\newtheorem{defn}[thm]{\protect\definitionname}
\theoremstyle{plain}
\newtheorem{prop}[thm]{\protect\propositionname}
\theoremstyle{plain}
\newtheorem{cor}[thm]{\protect\corollaryname}

\usepackage[a4paper]{geometry}
\usepackage{amsfonts}
\usepackage{amssymb}

\usepackage{color}
\definecolor{red}{rgb}{1,0,0} % red
\definecolor{green}{rgb}{0,1,0} % green
\definecolor{blue}{rgb}{0,0,1} % blue
\definecolor{darkblue}{rgb}{0,0,0.6}
\definecolor{darkred}{rgb}{0.6,0,0}

\usepackage[colorlinks,
            linkcolor=blue,
            anchorcolor=darkblue,
            citecolor=darkblue
           ]{hyperref}

\newcommand{\eps}{\varepsilon}  % greeks
  
\newcommand{\vf}{\varphi}

\newcommand{\md}{\mathrm{d}}   % differential operator
\newcommand{\vd}{\,\md}
\newcommand{\me}{\mathrm{e}}   % exponential

\newcommand{\pd}{\partial}   % partial derivative

\newcommand{\R}{\mathbb{R}}   % set of real numbers
\newcommand{\Dom}{G}

\newcommand{\PS}{\Omega}  % probability space
\newcommand{\E}{\mathbf{E}}  % expectation
\newcommand{\Prob}{\mathbf{P}}  % probability
\newcommand{\Filt}{\mathcal{F}}  % filtration
\newcommand{\Pred}{\mathcal{P}}  % predictable algebra
\newcommand{\BM}{w}  % Brownian motion

\newcommand{\cm}{\varpi}  % continuity modulus
\newcommand{\lap}{\Delta}

\newcommand{\bW}{\mathbb{W}}
\newcommand{\bL}{\mathbb{L}}

\makeatother

\usepackage{babel}
\providecommand{\assumptionname}{Assumption}
\providecommand{\corollaryname}{Corollary}
\providecommand{\definitionname}{Definition}
\providecommand{\examplename}{Example}
\providecommand{\lemmaname}{Lemma}
\providecommand{\propositionname}{Proposition}
\providecommand{\theoremname}{Theorem}

\begin{document}

\begin{frontmatter}{}

\title{$W^{2,p}$-solutions of parabolic SPDEs in general domains}

\author{Kai Du}

\address{Shanghai Center for Mathematical Sciences, Fudan University, Shanghai
200433, China}

\ead{kdu@fudan.edu.cn}
\begin{abstract}
The Dirichlet problem for a class of stochastic partial differential
equations is studied in Sobolev spaces. The existence and uniqueness
result is proved under certain compatibility conditions that ensure
the finiteness of $L^{p}(\Omega\times(0,T),W^{2,p}(G))$-norms of
solutions. The H\"older continuity of solutions and their derivatives
is also obtained by embedding.
\end{abstract}

\end{frontmatter}{}

\section{Introduction}

Given a domain $G\subset\R^{n}$ and a sequence of independent Wiener
processes $\BM^{k}$, let us consider the following stochastic partial
differential equation (SPDE)
\begin{equation}
\md u=(a^{ij}u_{x^{i}x^{j}}+b^{i}u_{x^{i}}+cu+f)\vd t+(\sigma^{ik}u_{x^{i}}+\nu^{k}u+g^{k})\vd\BM_{t}^{k}\label{eq:first}
\end{equation}
with $(t,x,\omega)\in(0,T]\times\Dom\times\PS$, where the leading
coefficients $a^{ij}(t,x,\omega)$ and $\sigma^{ik}(t,x,\omega)$
satisfy the \emph{strong parabolicity condition}: there are positive
numbers $\kappa$ and $K$ such that
\begin{equation}
\kappa|\xi|^{2}+\sigma^{ik}\sigma^{jk}\xi^{i}\xi^{j}\le2a^{ij}\xi^{i}\xi^{j}\le K|\xi|^{2}\quad\text{for all }\xi\in\R^{n}\text{ and }(t,x,\omega).\label{eq:parabolic}
\end{equation}
Einstein summation convention is used in this paper with $i,j=1,\dots,n$
and $k=1,2,\dots$. Such equations arise in many applications such
as nonlinear filtering, statistical physics, and so on (see \citet{da2014stochastic}
and references therein). The countable sum of stochastic integrals
in (\ref{eq:first}) lets it include equations driven by cylindrical
white noise (cf. \citet{walsh1986introduction,krylov1999analytic}).

The main goal of this paper is to obtain the solvability of parabolic
SPDEs in the space $L^{p}(\Omega\times(0,T),W^{2,p}(G))$ with natural
structural conditions, where $W^{2,p}(G)$ is a standard Sobolev space
with $p\ge2$. To explain our interest in this problem, let us recall
some well-known results from SPDE theory in this aspect. Under the
framework of Hilbert spaces $H^m(G)=W^{m,2}(G)$, \citet{krylov1981stochastic}
proved the existence and uniqueness of weak solutions for a large
class of parabolic SPDEs, and then they proved the smoothness of solutions
when $G=\R^{n}$. So far, the theory for the Cauchy problem is rather
complete and satisfactory: a comprehensive $L^{p}$-theory of parabolic
SPDEs in the whole space was developed by \citet{krylov1996l_p} in
Bessel potential spaces $H^{s,p}(\R^{n})$ (equivalent to $W^{s,p}(\R^{n})$
when $s$ is a natural number), and a solvability theory in H\"older
classes was constructed by \citet{mikulevicius2000cauchy,du2015cauchy}.
As far as general domains $G$ are concerned, one of the greatest
difficulties is how to handle the ``bad'' behaviour of derivatives
of solutions near the boundary. Indeed, unless certain compatibility
conditions are fulfilled, the derivatives of the solutions may blow
up near the boundary even in the one-dimensional case. As an example,
let us take a look at the following finding from \citet[Theorem~5.3]{Krylov2003Brownian}.
\begin{lem}[\citet{Krylov2003Brownian}]
There exists a $\lambda_{0}>0$ such that if $\lambda\in(0,\lambda_{0})$
and the function $u$ with $u(t,0)=0$ for all $t$ and $u(0,\cdot)\in C_{0}^{\infty}(0,\infty)$
satisfies the equation
\begin{align*}
\md u & =u_{xx}\vd t+\sqrt{2-\lambda}\,u_{x}\vd W_{t}\quad\text{on }(0,\infty)^{2},
\end{align*}
then there exists a dense subset $S\subset(0,\infty)$ such that for
all $s\in S$ and $\alpha>\me^{-\frac{1}{2\lambda}}$, it holds almost
surely (a.s.) that $\lim_{x\downarrow0}x^{-\alpha}u(s,x)=\infty$;
consequently, $\limsup_{x\downarrow0}|u_{x}(s,x)|=\infty$.
\end{lem}

\citet{flandoli1990dirichlet} proved the existence and uniqueness
of solutions of parabolic SPDEs in the Hilbert space $H^{2m+1}(G)$
under a long series of compatibility conditions (see Theorem 4.1 there).
\citet{Brzeniak1995Stochastic} solved the equations in the Besov
space $B_{p,2}^{1}(G)$ (whose elements have first-order weak derivatives) requiring
$\sigma$ to be sufficiently small. Both of them used semigroup method,
and the leading coefficients of their equations were deterministic.
Applying PDE techniques, \citet{krylov1994aw} developed a $W^{m,2}$-theory
of linear SPDEs in general smooth domains, where the equations could
have random coefficients; instead of the compatibility conditions,
he introduced Sobolev spaces with weights to control the blow-up of
derivatives of solutions near the boundary. This idea was adopted
to develop a weighted $L^{p}$-theory for parabolic SPDEs in general
domains, see \citet{krylov1994aw,krylov1999sobolev,kim2004p,kim2004stochastic}
among others. 
For more aspects of regularity theory for quasilinear SPDEs in domains, we refer to
\citet{denis2005p,zhang2006lp,van2012stochastic,van2012maximal,debussche2015re,Gerencs2017Boundary} and the references therein.

By relaxing the requirement on derivatives of solutions, the weighted
$L^{p}$-theory of SPDEs is successful in dealing with equations under
very general assumptions on the coefficients. Nevertheless, it is
still interesting enough to ask under which circumstances the solutions
of SPDEs lie in the normal Sobolev spaces, especially the space $W^{2,p}$
in which the solutions found are called \emph{strong solutions} in
classical PDE theory (cf. \citet{lieberman1996second}). This question
seems not to be answered by the weighted $L^{p}$-theory of SPDEs.
To find natural conditions, let us start with two examples as follows,
which show that, if there is no restriction on the boundary values
of coefficient $\sigma$, the second-order derivatives need not be
square integrable. 
\begin{example}
Let $u(t,x)$ with $t\in(0,\infty)$ and $x\in G=(0,1)$ be a solution
of the equation:
\begin{align*}
\md u(t,x) & =(u_{xx}(t,x)+f(t))\vd t+\sigma(x)u_{x}(t,x)\vd W_{t},\\
u(0,x) & =u(t,0)=u(t,1)=0,
\end{align*}
where $W$ is a one-dimensional Wiener process, $\sigma\in C^{2}(\bar{G})$
with $\sup_{G}|\sigma|<2$, and $f\in L^{2}(0,\infty)$ are not identically
zero. From $L^{2}$-theory of SPDEs (cf. \citet{krylov1981stochastic}),
the equation has a unique (nonzero) solution $u\in L^{2}(\Omega\times(0,\infty),H_{0}^{1}(G))$.
However, if $\sigma(0)\sigma(1)\neq0$, we can see that $u_{xx}\notin L^{2}(\Omega\times(0,T)\times G)$
for any $T>0$. Indeed, if $u_{xx}\in L^{2}(\Omega\times(0,T)\times G)$
for some $T>0$, then by embedding $u_{x}$ is continuous on $[0,T]\times\bar{G}$,
and according to the boundary condition, we have $\sigma(0)u_{x}(t,0)=\sigma(1)u_{x}(t,1)=0$.
Therefore, $v:=u_{x}\in H_{0}^{1}(G)$ satisfies (in the sense of
distribution)
\[
\md v=v_{xx}\md t+(\sigma v_{x}+\sigma_{x}v)\vd W_{t},\quad v(0,x)=v(t,0)=v(t,1)=0,
\]
which implies $u_{x}=v=0$, and furthermore, $u=0$ by the boundary
condition, yielding a contradiction. Thus, $u\notin L^{2}(\Omega\times(0,T),H^{2}(G))$
as long as $\sigma(0)\sigma(1)\neq0$.
\end{example}

\begin{example}
With $\sigma\in(-2,2)\backslash\{0\}$, $T>0$ and $G=(0,1)$ the
following equation
\[
\md u=(u_{xx}+x/\sigma)\vd t+(\sigma u_{x}-t)\vd W_{t},\quad u(0,x)=u(t,0)=u(t,1)=0
\]
has a unique solution $u\in L^{2}(\Omega\times(0,T),H_{0}^{1}(G))$
from $L^{2}$-theory of SPDEs. Suppose $u_{xx}\in L^{2}(\Omega\times(0,T),L^{2}(G))$.
Then by a similar argument as in the previous example, we have that
$v=u_{x}\in H^{1}(G)$ satisfies 
\[
\md v=(v_{xx}+1/\sigma)\vd t+\sigma v_{x}\vd W_{t},\quad v(0,x)=0,\quad v(t,0)=v(t,1)=t/\sigma.
\]
Solving this equation we have $u_{x}(t,x)=v(t,x)=t/\sigma$, which
is impossible given $u(t,0)=u(t,1)=0$. Therefore, $u\notin L^{2}(\Omega\times(0,T),H^{2}(G))$.
\end{example}

From the above, certain compatibility condition on the coefficient
$\sigma$ must be involved to ensure the second-order derivatives
of solutions lie in $L^{p}(\Omega\times(0,T),W^{2,p}(G))$. This issue
was first addressed by \citet{flandoli1990dirichlet}, where $p=2$
and the coefficients of equations depended only on $x$. For $p>2$
there seems be no result in the literature. In this note we propose
the following condition.
\begin{assumption}
\label{assu:compatib}The vectors $\sigma^{\cdot k}=(\sigma^{1k},\dots,\sigma^{nk})$
restricted on $\pd\Dom$ are tangent to $\pd\Dom$, namely,
\begin{equation}
\bm{n}(x)\cdot\sigma^{\cdot k}(t,x,\omega)=0,\quad k=1,2,\dots\label{eq:compatib}
\end{equation}
for all $x\in\partial G$ and all $(t,\omega)$, where $\bm{n}(x)$
is a unit normal vector of $\pd\Dom$ at $x$.
\end{assumption}

When considering zero boundary conditions, this assumption is quite
necessary for our goal according to our examples, and technically,
it gives the least condition on $\sigma$ to ensure that $\sigma^{ik}u_{x^{i}}$
vanishes on the boundary for all $u\in W^{2,p}(G)\cap W_{0}^{1,p}(G)$
and all $k$. Meanwhile, the free term $g$ must equal zero on the
boundary consequently, otherwise the second-order derivatives of solutions
of Eq. (\ref{eq:first}) may still blow up near the boundary, as was
illustrated in \citet[Example~1.2]{krylov1994aw}.

Assumption \ref{assu:compatib} and the boundary value restriction
of $g$ are all we need additionally to achieve our goal. Indeed,
the main result of this paper, Theorem \ref{thm:main} below, yields
that, under Assumption \ref{assu:compatib} along with other standard
conditions on the coefficients and on the domain, SPDE (\ref{eq:first})
with zero initial-boundary condition has a unique solution $u$ in
the space $L^{p}(\Omega\times(0,T),\Pred,W^{2,p}(G))$ for any given
$f\in L^{p}(\Omega\times(0,T),\Pred,L^{p}(G))$ and $g\in L^{p}(\Omega\times(0,T),\Pred,W_{0}^{1,p}(G;\ell^{2}))$,
where $\Pred$ is the predictable $\sigma$-field. The requirement
on the boundary value of $g$ is attracted into the space. By embedding
the solution and its derivatives are globally H\"older
continuous as long as $p>n+2$. 

It is worth noting that Assumption \ref{assu:compatib} has local
impact on the regularity of solutions; in other words, if (\ref{eq:compatib})
is satisfied only on a portion of $\partial G$, then the solutions
possess $W^{2,p}$-regularity and continuity near this portion. This
property is elaborated in Theorem \ref{thm:local} in the next section.

This paper is organized as follows. In the next section the main results
are stated after introducing some notation and assumptions. Section
3 is devoted to the proof of Theorem \ref{thm:main}, consisting of
four subsections: in Subsection 3.1 we obtain the existence, uniqueness
and estimates of the solution of a model equation in the half space;
in Subsection 3.2 we derive a priori estimates for general equations
in $\mathcal{C}^{2}$ domains; the existence and uniqueness of solutions
in the general case is proved in Subsection 3.3 with the help of the
method of continuity and the Banach fixed-point theorem; and in Subsection
3.4 we prove the continuity of solutions and their derivatives. Theorem
\ref{thm:local} is proved in the final section.

\section{Notation and main results\label{sec:Main-results}}

Let $(\PS,\Filt,\Filt_{t},\Prob)$ be a complete filtered probability
space carrying a sequence of independent Wiener process $\BM^{k}$,
and $\Pred$ the predictable $\sigma$-field generated by $\Filt_{t}$.
Let $\R^{n}$ be an $n$-dimensional Euclidean space of points $x=(x^{1},\dots,x^{n})$,
and
\[
\R_{+}^{n}=\{x=(x^{1},x'):x^{1}>0,\,x'=(x^{2},\dots,x^{n})\in\R^{n-1}\}.
\]
Denote $B_{\rho}(x)=\{y\in\R^{n}:|x-y|<\rho\}$ and $B_{\rho}=B_{\rho}(0)$. Let $G$ be a domain
in $\R^{n}$. The following definition is taken from \citet[Page~165]{Krylov2008Lectures}.
\begin{defn}
\label{def:domain} We write $\Dom\in\mathcal{C}^{2}$ if there are
positive constants $K_{0}$ and $\rho_{0}$ such that for each $z\in\pd\Dom$
there exists a one-to-one map $\psi$ from $B_{\rho_{0}}(z)$ to a
domain $U^{z}\subset\R^{n}$ such that
\begin{enumerate}
\item $\psi(z)=0$ and $U_{+}^{z}:=\psi(B_{\rho_{0}}(z)\cap\Dom)\subset\R_{+}^{n}$,
\item $\psi(B_{\rho_{0}}(z)\cap\pd\Dom)=U^{z}\cap\{y\in\R^{n}:y^{1}=0\}$,
\item $\psi\in C^{2}(\bar{B}_{\rho_{0}}(z))$, $\psi^{-1}\in C^{2}(\bar{U}^{z})$,
and $\|\psi\|_{C^{2}}+\|\psi^{-1}\|_{C^{2}}\le K_{0}$.
\end{enumerate}
We say that the diffeomorphism $\psi$ flattens the boundary near
$z$.
\end{defn}

Fix real numbers $T>0$ and $p\ge2$ in this paper. 

For $m\ge0$ we let $W^{m,p}(G)$ and $W_{0}^{1,p}(G)$ be the usual
Sobolev spaces (cf. \citet{adams2003sobolev}), and $W^{m,p}(G;\ell^{2})$
and $W_{0}^{1,p}(G;\ell^{2})$ the corresponding spaces of $\ell^{2}$-valued
functions. Denote
\[
W_{\circ}^{m,p}(G)=W^{m,p}(G)\cap W_{0}^{1,p}(G),\quad m\ge1,
\]
and $W_{\circ}^{0,p}(\cdot)=W^{0,p}(\cdot)=L^{p}(\cdot)$. Denote
by $W_{{\rm loc}}^{m,p}(G)$ the space of all functions $u$ such
that $u\in W^{m,p}(G')$ for any $G'\subset G$ with $\mathrm{dist}(G',\partial G)>0$.

For random functions, we define 
\begin{align*}
\bW^{m,p}(G,\tau) & =L^{p}(\PS\times(0,\tau),\Pred,W^{m,p}(G)),\quad\bW^{m,p}(G)=\bW^{m,p}(G,T),\\
\bW_{\circ}^{m,p}(G,\tau) & =L^{p}(\PS\times(0,\tau),\Pred,W_{\circ}^{m,p}(G)),\quad\bW_{\circ}^{m,p}(G)=\bW_{\circ}^{m,p}(G,T),
\end{align*}
and analogously, $\bW^{m,p}(G,\tau;\ell^{2})$, $\bW^{m,p}(G;\ell^{2})$,
$\bW_{\circ}^{m,p}(G;\ell^{2}),\bW_{{\rm loc}}^{m,p}(G)$, etc. Denote
$\bL^{p}(\cdot)=\bW^{0,p}(\cdot)=\bW_{\circ}^{0,p}(\cdot)$.

The understanding of solutions of SPDEs is implied in the following
definition of a functional space for solutions (cf. \citet{krylov1999analytic}).
\begin{defn}
\label{def:W2p}For a positive integer $m$, by $\mathcal{W}_{\circ}^{m,p}(G)$
we denote the space of all functions $u\in\bW_{\circ}^{m,p}(G)$ such
that 
\[
u(0,\cdot)\in L^{p}(\PS,\Filt_{0},W_{\circ}^{m-2/p,p}(\Dom))
\]
and for some $u_{{\rm D}}\in\bW^{m-2,p}(G)$ and $u_{{\rm S}}\in\bW_{\circ}^{m-1,p}(G;\ell^{2})$,
the equation $\md u=u_{{\rm D}}\vd t+u_{{\rm S}}^{k}\vd\BM_{t}^{k}$
holds in the sense of distributions, namely, for all $\phi\in C_{0}^{\infty}(G)$,
\[
(u(t,\cdot),\phi)=(u(0,\cdot),\phi)+\int_{0}^{t}(u_{{\rm D}}(s,\cdot),\phi)\vd s+\int_{0}^{t}(u_{{\rm S}}^{k}(s,\cdot),\phi)\vd\BM_{s}^{k}
\]
for all $t\le T$ with probability $1$.
\end{defn}

Now we consider the following semilinear equation
\begin{equation}
\md u=(a^{ij}u_{x^{i}x^{j}}+f(t,x,u))\vd t+(\sigma^{ik}u_{x^{i}}+g^{k}(t,x,u))\vd\BM_{t}^{k}\label{eq:main}
\end{equation}
with the initial-boundary condition
\begin{equation}
\Big\{\!\begin{array}{ll}
u(t,x)=0, & x\in\pd\Dom,\ t\ge0;\\
u(0,x)=u_{0}(x), & x\in\Dom.
\end{array}\label{eq:bdcondition}
\end{equation}
The following conditions on the given data are quite standard (cf.
\citet{krylov1999analytic}).
\begin{assumption}
\label{assu:continuity}The functions $a^{ij}=a^{ji}$ and $\sigma^{ik}$
are real valued and $\Pred\times\mathcal{B}(\Dom)$-measurable and
satisfy the strong parabolicity condition (\ref{eq:parabolic}), and
there are a number $L>0$ and a continuous and increasing function
$\cm(\cdot)$ with $\cm(0)=0$ such that
\[
|a^{ij}(t,x)-a^{ij}(t,y)|\le\cm(|x-y|),\quad\|\sigma^{i\cdot}(t,x)-\sigma^{i\cdot}(t,y)\|_{\ell^{2}}\le L|x-y|,
\]
 for all $(t,\omega)$, all $x,y\in\bar{\Dom}$, and all $i,j=1,\dots,n$.
\end{assumption}

\begin{assumption}
\label{assu:fg}{(a)} For any $u\in W_{\circ}^{2,p}(\Dom)$,
the functions $f(\cdot,\cdot,u)$ and $g(\cdot,\cdot,u)$ are predictable
as functions taking values in $L^{p}(G)$ and $W_{\circ}^{1,p}(G;\ell^{2})$,
respectively.

{(b)} $f(\cdot,\cdot,0)\in\bL^{p}(\Dom)$ and $g(\cdot,\cdot,0)\in\bW_{\circ}^{1,p}(\Dom;\ell^{2})$.

{(c)} For any $\eps>0$, there is a $K_{\eps}\ge0$ such that
for any $u,v\in W_{\circ}^{2,p}(\Dom)$, $t$, $\omega$, we have
\[
\begin{aligned}\|f(t,\cdot,u)-f(t,\cdot,v)\|_{L^{p}(G)} & +\|g(t,\cdot,u)-g(t,\cdot,v)\|_{W^{1,p}(G;\ell^{2})}\\
 & \le\eps\|u-v\|_{W^{2,p}(G)}+K_{\eps}\|u-v\|_{L^{p}(G)}.
\end{aligned}
\]
\end{assumption}

The main result of this paper is the following theorem.
\begin{thm}
\label{thm:main}Let $\Dom\in\mathcal{C}^{2}$ and Assumptions \ref{assu:compatib},
\ref{assu:continuity} and \ref{assu:fg} be satisfied. Then we have
that

\emph{(i)} for any $u_{0}(\cdot)\in L^{p}(\PS,\Filt_{0},W_{\circ}^{2-2/p,p}(\Dom))$
Dirichlet problem (\ref{eq:main})--(\ref{eq:bdcondition}) admits
a unique solution $u\in\mathcal{W}_{\circ}^{2,p}(G)$;

\emph{(ii)} the solution satisfies the estimate
\begin{equation}
\|u\|_{\bW^{2,p}(G)}^{p}\le C\bigl(\|f(\cdot,\cdot,0)\|_{\bL^{p}(G)}^{p}+\|g(\cdot,\cdot,0)\|_{\bW^{1,p}(G;\ell^{2})}^{p}+\E\|u_{0}\|_{W^{2-2/p,p}(G)}^{p}\bigr),\label{eq:main-est}
\end{equation}
where the constant $C$ depends only on $\kappa,K,n,p,T,K_{0},\rho_{0},L,$
and the functions $\cm(\cdot)$ and $K_{\eps}$;

\emph{(iii)} when $p>\max\{2,(n+2)/2\}$, $u\in L^{p}(\Omega,C^{\alpha/2,\alpha}([0,T]\times\bar{G}))$
for any $\alpha\in(0,\frac{2p-n-2}{2p})$, and when $p>n+2$, $u_{x}\in L^{p}(\Omega,C^{\beta/2,\beta}([0,T]\times\bar{G}))$
for any $\beta\in(0,\frac{p-n-2}{2p})$.
\end{thm}

The H\"older space $C^{\alpha/2,\alpha}([0,T]\times\bar{G})$ is
defined in the standard way (cf. \citet{krylov1996lectures}), which
contains all continuous functions $u:[0,T]\times\bar{G}\to\R$ such
that
\[
\|u\|_{C^{\alpha/2,\alpha}([0,T]\times\bar{G})}=\sup_{[0,T]\times\bar{G}}|u|+\sup_{(t,x)\neq(s,y)}\frac{|u(t,x)-u(s,y)|}{|t-s|^{\alpha/2}+|x-y|^{\alpha}}<\infty.
\]

In the literature the domain $G$ was usually assumed to be bounded
(unless it is the whole space or a half space), but here it can be
unbounded, and a detailed argument will be presented in our proof
to address unbounded domains (see Subsection 3.2 below). Moreover,
the above theorem still holds true if the terminal time $T$ is replaced
by any stopping time $\tau\le T$ as in \citet{krylov1999analytic,kim2004p}.
A simple way to do this is to zero extend the functions $f$ and $g$
after time $\tau$ until $T$ and solve the problem in the time period
$[0,T]$.

By interpolation it is easily checked that the linear equation (\ref{eq:first})
fits the assumptions of Theorem (\ref{thm:main}) provided that $f\in\bL^{p}(\Dom)$
and $g\in\bW_{\circ}^{1,p}(\Dom;\ell^{2})$ along with the following
condition.
\begin{assumption}
\label{assu:bcnu}The functions $b^{i}$, $c$, $\nu^{k}$ are real
valued and $\Pred\times\mathcal{B}(\Dom)$-measurable, and $|b^{i}|$,
$|c|$, $\|\nu\|_{\ell^{2}}$, $\|\nu_{x}\|_{\ell^{2}}$ are uniformly
bounded on $[0,T]\times\bar{G}\times\Omega$.
\end{assumption}

Even if the compatibility condition (\ref{eq:compatib}) is satisfied
only on a portion of $\partial G$, it is still possible to obtain
local regularity of solutions near this portion. The main issue here
is that the solution may not lie in $\mathcal{W}_{\circ}^{2,p}(G)$
or even not exist. Fortunately, when $G$ is bounded (or a half space)
and the equation is linear, the Dirichlet problem can be solved in
the weighted Sobolev space $\mathfrak{H}_{p,\theta}^{2}(G)$ by means
of the main results in \citet{kim2004stochastic}; the space $\mathfrak{H}_{p,\theta}^{2}(G)$
that $\mathcal{W}_{\circ}^{2,p}(G)$ can be embedded to was introduced
by \citet{krylov1999sobolev,Lototsky2000Sobolev}, based on delicately
selected weights. With this observation, we formulate the local regularity
result into the following theorem by assuming the existence of solutions
without thorough verification of conditions, and for simplicity but
without loss of essence, we consider the linear equation (\ref{eq:first}). 
\begin{thm}
\label{thm:local}Let $\Gamma$ be an open subset of $\partial G$
and (\ref{eq:compatib}) satisfied at each point $x\in\Gamma$. Let
Assumptions \ref{assu:continuity} and \ref{assu:bcnu} be satisfied
and $|a_{x}^{ij}|$ dominated by the constant $L$. Suppose that $u\in\bL^{p}(G)\cap\bW_{{\rm loc}}^{2,p}(G)$
with $u(0,\cdot)\in L^{p}(\PS,\Filt_{0},W_{\circ}^{2-2/p,p}(\Dom))$
satisfies Eq. (\ref{eq:first}) with $u|_{\partial G}=0$ for given
$f\in\bL^{p}(G)$ and $g\in\bW^{1,p}(G)$ with $g|_{\Gamma}=0$. Then
for any bounded domain $G'\subset G$ with $\mathrm{dist}(G',\partial G\backslash\Gamma)>0$,
we have $u\in\bW^{2,p}(G')$. Moreover, $u\in L^{p}(\Omega,C^{\alpha/2,\alpha}([0,T]\times\bar{G}'))$
for any $\alpha\in(0,\frac{2p-n-2}{2p})$ when $p>\max\{2,(n+2)/2\}$,
and $u_{x}\in L^{p}(\Omega,C^{\beta/2,\beta}([0,T]\times\bar{G'}))$
for any $\beta\in(0,\frac{p-n-2}{2p})$ when $p>n+2$.
\end{thm}

In the above theorem the assumption that the solution lies in $\bL^{p}(G)\cap\bW_{{\rm loc}}^{2,p}(G)$
is not restrictive: on the one hand, a function in the space $\mathfrak{H}_{p,\theta}^{2}(G)$
naturally belongs to $\bW_{{\rm loc}}^{2,p}(G)$; on the other hand,
the property $u\in\bL^{p}(G)$ can be derived from the other assumptions
of the theorem with the help of It\^o's formula, at least when $G$
is bounded. Requiring $a^{ij}(t,\cdot)\in C^{0,1}(G)$ allows us to
write the equation into the divergence form that helps us prove $u\in \bW^{1,p}(G')$
as an important intermediate step. We remark that $u\in\mathfrak{H}_{p,\theta}^{2}(G)$
does not always imply $u\in\bW^{1,p}(G)$ (cf. \citet{kim2004p}).

\section{Proof of Theorem \ref{thm:main}}

\subsection{Model equations in a half space}

Let $G=\R_{+}^{n}$ in this subsection. In the first step we consider
the equations on $[0,T]\times\R_{+}^{n}$ with coefficients independent
of $x$. 
\begin{prop}
\label{prop:half-constant}Let $a^{ij}=a^{ij}(t)$ and $\sigma^{ik}=\sigma^{ik}(t)$
be predictable processes and satisfy (\ref{eq:parabolic}). Assume
that
\[
\sigma^{1\cdot}=(\sigma^{11},\sigma^{12},\dots)\equiv0,\quad\forall(t,\omega)\in[0,T]\times\PS.
\]
Consider the Dirichlet problem
\begin{equation}
\bigg\{\begin{aligned} & \md u=(a^{ij}u_{x^{i}x^{j}}+f)\vd t+(\sigma^{ik}u_{x^{i}}+g^{k})\vd\BM_{t}^{k},\\
 & u|_{t=0}=0,\quad u|_{x^{1}=0}=0.
\end{aligned}
\label{eq:half-space-1}
\end{equation}
Then, \emph{(i)} for $f\in\bL^{p}(\R_{+}^{n})$ and $g\in\bW_{\circ}^{1,p}(\R_{+}^{n};\ell^{2})$,
(\ref{eq:half-space-1}) has a unique solution $u\in\mathcal{W}_{\circ}^{2,p}(\R_{+}^{n})$,
and
\begin{equation}
\|u\|_{\bW^{2,p}(\R_{+}^{n})}\le C\bigl(\|f\|_{\bL^{p}(\R_{+}^{n})}+\|g\|_{\bW^{1,p}(\R_{+}^{n};\ell^{2})}\bigr);\label{eq:est-half}
\end{equation}
 \emph{(ii)} if $g\in\bL^{p}(\R_{+}^{n};\ell^{2})$ and $f=f^{0}+c^{i}F_{x^{i}}$
with $f^{0},F\in\bL^{p}(\R_{+}^{n})$ and $c^{i}\in\bL^{\infty}(G)$,
then (\ref{eq:half-space-1}) has a unique solution $u\in\mathcal{W}_{\circ}^{1,p}(\R_{+}^{n})$,
and
\begin{equation}
\|u\|_{\bW^{1,p}(\R_{+}^{n})}\le C\bigl(\|(f^{0},F)\|_{\bL^{p}(\R_{+}^{n})}+\|g\|_{\bL^{p}(\R_{+}^{n};\ell^{2})}\bigr);\label{eq:est-half-3}
\end{equation}
where the constant $C$ depends only on $\kappa,K,n,p,$ $T$, and
additionally on $\|c^{i}\|_{\bL^{\infty}}$ for \emph{(ii)}.
\end{prop}

\begin{proof}
The proofs of (i) and (ii) are quite similar, so we only present the
proof of (i) in details. Consider the following equation
\begin{equation}
\md\hat{u}=K\lap\hat{u}\vd t+(\sigma^{ik}\hat{u}_{x^{i}}+g^{k})\vd\BM_{t}^{k}\quad\text{on }(0,T]\times\R_{+}^{n}\label{eq:lap-halp}
\end{equation}
with zero initial-boundary condition. Obviously, this equation is
also strongly parabolic. Define the odd continuation of $g$, i.e.,
\begin{equation}
g(x^{1},x'):=-g(-x^{1},x'),\quad\forall x^{1}<0,\,x'\in\R^{n-1}.\label{eq:odd-ext}
\end{equation}
As $g\in\bW_{\circ}^{1,p}(\R_{+}^{n};\ell^{2})$, the continued function
$g$ belongs to $\bW_{\circ}^{1,p}(\R^{n};\ell^{2})$. By Theorem
5.1 in \citet{krylov1999analytic}, there exists a unique solution
$\hat{u}\in\mathcal{W}_{\circ}^{2,p}(\R^{n})$ of (\ref{eq:lap-halp})
considered in the whole $\R^{n}$ with zero initial condition. From
the uniqueness, $\hat{u}(t,x)=\hat{u}(t,x^{1},x')$ is odd with respect
to $x^{1}$, so $\hat{u}(t,x)=0$ for $x^{1}=0$, which means that
$\hat{u}$ restricted in $\R_{+}^{n}$ satisfies (\ref{eq:lap-halp})
with zero initial-boundary condition, and $\hat{u}\in\mathcal{W}_{\circ}^{2,p}(\R_{+}^{n})$.
Also from Theorem 5.1 in \citet{krylov1999analytic}, we have the
following estimate
\begin{equation}
\|\hat{u}\|_{\bW^{2,p}(\R_{+}^{n})}\le C\|g\|_{\bW^{1,p}(\R_{+}^{n};\ell^{2})},\label{eq:hat-u}
\end{equation}
where the constant $C$ depends only on $\kappa,K,n,p$ and $T$.

Define a stochastic process $\xi_{t}=(0,\xi_{t}^{2},\dots,\xi_{t}^{n})$
with
\[
\xi_{t}^{i}=\int_{0}^{t}\sigma^{ik}(s)\vd\BM_{s}^{k},\quad i=2,\dots,n.
\]
It is easily seen that for each $x\in\R_{+}^{n}$ the process $x\pm\xi_{t}$
always stays in $\R_{+}^{n}$. Moreover, for any given $\tilde{f}\in\bL^{p}(\R_{+}^{n})$,
the random translation $\tilde{f}(t,x-\xi_{t})$ as a function of
$(t,x,\omega)$ also lies in $\bL^{p}(\R_{+}^{n})$, and
\[
\|\tilde{f}(\omega)\|_{L^{p}((0,T)\times\R_{+}^{n})}=\|\tilde{f}(\cdot,\cdot-\xi_{\cdot}(\omega),\omega)\|_{L^{p}((0,T)\times\R_{+}^{n})}.
\]
Consider the following random partial differential equation (PDE)
\begin{equation}
\begin{aligned}\partial_{t}v & =\Bigl(a^{ij}-\frac{1}{2}\sigma^{ik}\sigma^{jk}\Big)v_{x^{i}x^{j}}+\tilde{f}(t,x-\xi_{t}),\quad\text{on }(0,T]\times\R_{+}^{n},\\
v|_{t=0} & =0,\quad v|_{x^{1}=0}=0.
\end{aligned}
\label{eq:randomPDE}
\end{equation}
Due to (\ref{eq:parabolic}), this PDE is strongly parabolic. Moreover,
$\tilde{f}(t,x-\xi_{t}(\omega),\omega)$ as a function of $(t,x)$
belongs to $L^{p}((0,T)\times\R_{+}^{n})$ for almost every $\omega$.
So by the classical PDE theory (cf. \citet[Theorem~7.32]{lieberman1996second}),
problem (\ref{eq:randomPDE}) has a unique strong solution 
\[
v(\cdot,\cdot,\omega)\in L^{p}((0,T),W_{\circ}^{2,p}(\R_{+}^{n}))\times C([0,T],L^{p}(\R_{+}^{n}))
\]
for almost every $\omega$, and $v(\cdot,\cdot,\omega)$ satisfies
the estimates
\[
\|v(\cdot,\cdot,\omega)\|_{L^{p}((0,T),W^{2,p}(\R_{+}^{n}))}^{p}\le C\|\tilde{f}(\cdot,\cdot,\omega)\|_{L^{p}((0,T)\times\R_{+}^{n})}^{p},
\]
where the constant $C$ depends only on $\kappa,K,n,p$ and $T$,
but not on $\omega$. Thus, one has $v\in\mathcal{W}_{\circ}^{2,p}(\R_{+}^{n})$,
and taking mathematical expectation on the above estimate yields
\begin{equation}
\|v\|_{\bW^{2,p}(\R_{+}^{n})}^{p}\le C\|\tilde{f}\|_{\bL^{p}(\R_{+}^{n})}^{p}.\label{eq:v01}
\end{equation}
Now applying the It\^o-Wentzell formula obtained in \citet{Krylov2011On}
to $\tilde{u}(t,x):=v(t,x+\xi_{t})$, one can check that $\tilde{u}\in\mathcal{W}_{\circ}^{2,p}(\R_{+}^{n})$,
and it solves the problem
\begin{equation}
\md\tilde{u}=(a^{ij}\tilde{u}_{x^{i}x^{j}}+\tilde{f})\vd t+\sigma^{ik}\tilde{u}_{x^{i}}\vd\BM_{t}^{k},\quad\tilde{u}|_{t=0}=0,\quad\tilde{u}|_{x^{1}=0}=0,\label{eq:proof003}
\end{equation}
and satisfies the estimate
\begin{equation}
\|\tilde{u}\|_{\bW^{2,p}(\R_{+}^{n})}=\|v\|_{\bW^{2,p}(\R_{+}^{n})}\le C\|\tilde{f}\|_{\bL^{p}(\R_{+}^{n})}.\label{eq:tilde-u}
\end{equation}
On the other hand, we remark that $\tilde{u}\in\mathcal{W}_{\circ}^{2,p}(\R_{+}^{n})$
is a solution to (\ref{eq:proof003}) if and only if $v(t,x)=\tilde{u}(t,x-\xi_{t})$
is the solution to (\ref{eq:randomPDE}); as the latter has a unique
solution, the solution of (\ref{eq:proof003}) is also unique.

Define $u=\hat{u}+\tilde{u}\in\mathcal{W}_{\circ}^{2,p}(\R_{+}^{n})$.
It follows from equations (\ref{eq:lap-halp}) and (\ref{eq:proof003})
that
\[
\md u=[a^{ij}u_{x^{i}x^{j}}+(K\delta^{ij}-a^{ij})\hat{u}_{x^{i}x^{j}}+\tilde{f}]\vd t+(\sigma^{ik}u_{x^{i}}+g^{k})\vd\BM_{t}^{k},
\]
where $\delta^{ij}$ is the Kronecker delta. With
\begin{equation}
\tilde{f}=f-(K\delta^{ij}-a^{ij})\hat{u}_{x^{i}x^{j}},\label{eq:pf005}
\end{equation}
it is seen that $u$ is a solution to problem (\ref{eq:half-space-1}),
so the existence part is proved. Moreover, from estimates (\ref{eq:hat-u})
and (\ref{eq:tilde-u}), the obtained $u$ satisfies
\begin{equation}
\begin{aligned}\|u\|_{\bW^{2,p}(\R_{+}^{n})} & \le\|\hat{u}+\tilde{u}\|_{\bW^{2,p}(\R_{+}^{n})}\\
 & \le C\bigl(\|g\|_{\bW^{1,p}(\R_{+}^{n};\ell^{2})}+\|f-(K\delta^{ij}-a^{ij})\hat{u}_{x^{i}x^{j}}\|_{\bL^{p}(\R_{+}^{n})}\bigr)\\
 & \le C\bigl(\|f\|_{\bL^{p}(\R_{+}^{n})}+\|g\|_{\bW^{1,p}(\R_{+}^{n};\ell^{2})}+\|\hat{u}_{xx}\|_{\bL^{p}(\R_{+}^{n})}\bigr)\\
 & \le C\bigl(\|f\|_{\bL^{p}(\R_{+}^{n})}+\|g\|_{\bW^{1,p}(\R_{+}^{n};\ell^{2})}\bigr),
\end{aligned}
\label{eq:pf004}
\end{equation}
where $C=C(\kappa,K,n,p,T)$. 

To prove the uniqueness, we let $u^{*}\in\mathcal{W}_{\circ}^{2,p}(\R_{+}^{n})$
be any solution of (\ref{eq:half-space-1}), $\hat{u}$ be the solution
of (\ref{eq:lap-halp}) determined by $g$. Then $u^{*}-\hat{u}$
satisfies (\ref{eq:proof003}) with $\tilde{f}$ given by (\ref{eq:pf005}),
so by uniqueness (for problem (\ref{eq:proof003})) we have $u^{*}-\hat{u}=\tilde{u}$,
which means $u=u^{*}$. The uniqueness part is also proved, and the
estimate (\ref{eq:est-half}) follows from (\ref{eq:pf004}) immediately.
The proof is complete.
\end{proof}

\subsection{A priori estimates}

In the following result we obtain a priori estimates for linear equations
in general domains. We adapt the technique of straightening (the boundary)
and partitioning (the domain) from PDE theory (cf. \citet{gilbarg2001elliptic,Krylov2008Lectures}).
The new difficulties here are due to the compatibility conditions
and the (possible) unbounded domains. Recall that $u\in\mathcal{W}_{\circ}^{2,p}(G)$
implies $u(t,\cdot)=0$ on the boundary $\pd\Dom.$
\begin{prop}
\label{prop:apriori}Let $\Dom\in\mathcal{C}^{2}$ and Assumptions
\ref{assu:compatib} and \ref{assu:continuity} be satisfied. Suppose
that $u\in\mathcal{W}_{\circ}^{2,p}(G)$ with $u(0,\cdot)=0$ satisfies
the equation
\begin{equation}
\md u=(a^{ij}u_{x^{i}x^{j}}+f)\vd t+(\sigma^{ik}u_{x^{i}}+g^{k})\vd\BM_{t}^{k}\label{eq:general-dom}
\end{equation}
for some $f\in\bL^{p}(G)$ and $g\in\bW_{\circ}^{1,p}(G)$. Then we
have
\begin{equation}
\|u\|_{\bW^{2,p}(G)}\le C\bigl(\|f\|_{\bL^{p}(G)}+\|g\|_{\bW^{1,p}(G;\ell^{2})}\bigr),\label{eq:apriori}
\end{equation}
where the constant $C$ depends only on $\kappa,K,n,p,T,K_{0},\rho_{0},L,$
and the functions $\cm(\cdot)$.
\end{prop}

\begin{proof}
Fix a $z\in\pd\Dom$ and take the objects associated with $z$ from
Definition \ref{def:domain}. For a function $h$ defined in $B_{\rho_{0}}(z)\cap\Dom$,
we introduce
\[
\tilde{h}(y)=h\circ\psi^{-1}(y)=h(\psi^{-1}(y))\quad\forall\,y\in U_{+}^{z}.
\]
Obviously, $h(x)=\tilde{h}\circ\psi(x)$. In what follows, we keep
the relation
\[
y=\psi(x)\quad\text{for }x\in B_{\rho_{0}}(z)\cap\Dom,
\]
which implies that $h(x)=\tilde{h}(y)$. 

For the sake of convenience, we denote $h_{i}=\partial_{i}h$ in this
subsection to be the partial derivative of a function $u$ with respect
to the $i$-th spatial variable. Then for $h\in W^{2,p}(B_{\rho_{0}}(z)\cap\Dom)$
we have
\[
\begin{aligned}h_{i}(x) & =\psi_{i}^{r}(x)\tilde{h}_{r}(y),\\
h_{ij}(x) & =\psi_{i}^{r}(x)\psi_{j}^{s}(x)\tilde{h}_{rs}(y)+\psi_{ij}^{r}(x)\tilde{h}_{r}(y).
\end{aligned}
\]
The following result is taken from Lemma 8.3.4 in \citet{Krylov2008Lectures}.
\begin{lem}
\label{lem:equivalent}$h\in W^{k,p}(B_{\rho_{0}}(z)\cap\Dom)$ if
and only if $\tilde{h}\in W^{k,p}(U_{+}^{z})$ for $k=0,1,2$. Moreover,
\[
C^{-1}\|h\|_{W^{k,p}(B_{\rho_{0}}(z)\cap\Dom)}\le\|\tilde{h}\|_{W^{k,p}(U_{+}^{z})}\le C\|h\|_{W^{k,p}(B_{\rho_{0}}(z)\cap\Dom)}
\]
with $C=C(n,p,K_{0})$. 
\end{lem}

Let $\eta\in C_{0}^{\infty}(\R^{n})$ such that $\eta(x)=1$ for $|x|\le\rho_{0}/2$
and $\eta(x)=0$ for $|x|\ge3\rho_{0}/4$. With $y=\psi(x)$ (only
for $x\in B_{\rho_{0}}(z)\cap\Dom$) we define
\[
\begin{aligned}\eta^{z}(x) & =\eta(x-z),\quad\tilde{\eta}^{z}(y)=\eta^{z}(x),\\
\tilde{a}^{rs}(t,y) & =a^{ij}(t,x)\psi_{i}^{r}(x)\psi_{j}^{s}(x)\tilde{\eta}^{z}(y)+K\delta^{rs}[1-\tilde{\eta}(y)],\\
\tilde{\sigma}^{rk}(t,y) & =\sigma^{ik}(t,x)\psi_{i}^{r}(x)\tilde{\eta}(y).
\end{aligned}
\]
Formally speaking, $\tilde{a}^{rs}(t,y)$ and $\tilde{\sigma}^{rk}(t,y)$
are not defined for $y\notin\bar{U}^{z}$, but we may set $\tilde{\eta}^{z}(y)=0$
for those $y$ and the corresponding terms to be zero, then $\tilde{a}^{rs}(t,y)$
and $\tilde{\sigma}^{rk}(t,y)$ are well-defined for all $y\in\R_{+}^{n}$.
From Lemma 8.3.6 in \citet{Krylov2008Lectures}, we have
\begin{lem}
\emph{\label{lem:tilde-a-sigma}(i)} For any $y,y_{1},y_{2}\in\R_{+}^{n}$
and $(t,\omega)\in[0,T]\times\PS$,
\[
\begin{aligned}|\tilde{a}^{rs}(t,y)|+\|\tilde{\sigma}^{r\cdot}(t,y)\|_{\ell^{2}} & \le\tilde{K}(n,K,K_{0}),\\
|\tilde{a}^{rs}(t,y_{1})-\tilde{a}^{rs}(t,y_{2})| & \le\tilde{\cm}(|y_{1}-y_{2}|),\\
\|\tilde{\sigma}^{r\cdot}(t,y_{1})-\tilde{\sigma}^{r\cdot}(t,y_{2})\|_{\ell^{2}} & \le\tilde{L}(n,K,K_{0},L),
\end{aligned}
\]
where $\tilde{\cm}(\cdot)$ is a modulus of continuity determined
only by $\cm(\cdot),n,K$ and $K_{0}$.

\emph{(ii)} There is a constant $\tilde{\kappa}=\tilde{\kappa}(n,\kappa,K_{0})>0$
such that
\[
(2\tilde{a}^{rs}(t,y)-\tilde{\sigma}^{rk}(t,y)\tilde{\sigma}^{sk}(t,y))\xi^{i}\xi^{j}\ge\tilde{\kappa}|\xi|^{2}
\]
for all $\xi=(\xi^{1},\dots,\xi^{n})\in\R^{n}$ and all $(t,y)\in[0,T]\times\R_{+}^{n}$.
\end{lem}

Let $\rho\in(0,\rho_{0}\wedge1]$ be a constant to be specified later,
and take a nonnegative function $\zeta\in C_{0}^{\infty}(\R^{n})$
such that $\zeta(x)=1$ for $|x|\le\rho/4$ and $\zeta(x)=0$ for
$|x|\ge\rho/2$. Set 
\begin{equation}
\zeta^{z}(x)=\zeta(x-z)\quad\text{and}\quad\tilde{\zeta}^{z}(y)=\zeta^{z}(x)=\zeta^{z}(\psi^{-1}(y)).\label{eq:pf020}
\end{equation}
It is easily checked that $\rho|\zeta_{x}|+\rho^{2}|\zeta_{xx}|\le C(n)$.

Let $u\in\mathcal{W}_{\circ}^{2,p}(\Dom)$ be a solution to Eq. (\ref{eq:general-dom})
with $u(0,\cdot)=0$. Define
\begin{equation}
\tilde{u}^{z}(t,y)=\Big\{\begin{array}{ll}
\tilde{\zeta}^{z}(y)u(t,x) & \text{for }y\in\bar{U}_{+}^{z},\\
0 & \text{elsewhere}.
\end{array}\label{eq:pr018}
\end{equation}
A direct computation gives that the function $\tilde{u}^{z}\in\mathcal{W}_{\circ}^{2,p}(\R_{+}^{n})$,
whose support lies in $\psi(B_{\rho/2}(z))\cap\bar{U}_{+}^{z}$, satisfies
the following equation
\begin{equation}
\begin{aligned}\md\tilde{u}^{z}(t,y) & =(\tilde{a}^{rs}(t,0)\tilde{u}_{rs}^{z}(t,y)+\hat{f}^{z}(t,y))\vd t+(\tilde{\sigma}^{rk}(t,0)\tilde{u}_{r}^{z}(t,y)+\hat{g}^{z,k}(t,y))\vd\BM_{t}^{k}\\
\tilde{u}^{z}|_{t=0} & =0,\quad\tilde{u}^{z}|_{y^{1}=0}=0,
\end{aligned}
\label{eq:tilde-uz}
\end{equation}
where
\begin{align}
\hat{f}^{z}(t,y) & =[\tilde{a}^{rs}(t,y)-\tilde{a}^{rs}(t,0)]\tilde{u}_{rs}^{z}(t,y)+\tilde{\zeta}^{z}(y)f(t,x)-\tilde{a}^{rs}(t,y)\tilde{\zeta}_{rs}^{z}(y)\tilde{u}(t,y)\nonumber \\
 & \quad-\tilde{a}^{rs}(t,y)\tilde{\zeta}_{s}^{z}(y)\tilde{u}_{r}(t,y)+a^{ij}(t,x)\psi_{ij}^{r}(x)\tilde{\zeta}^{z}(y)\tilde{u}_{r}(t,y),\nonumber \\
\hat{g}^{z,k}(t,y) & =[\tilde{\sigma}^{rk}(t,y)-\tilde{\sigma}^{rk}(t,0)]\tilde{u}_{r}^{z}(t,y)+\tilde{\zeta}^{z}(y)g^{k}(t,x)-\tilde{\sigma}^{rk}(t,y)\tilde{\zeta}_{r}^{z}(y)\tilde{u}^{z}(t,y).\label{eq:pr010}
\end{align}
To apply Proposition \ref{prop:apriori} to Eq. (\ref{eq:tilde-uz}),
we need to verify the following conditions:
\begin{align}
 & \tilde{\sigma}^{1\cdot}(t,0,y')=0\ \ \forall\,y'\in\R_{+}^{n},\label{eq:pf-sigma}\\
 & \hat{f}^{z}\in\bL^{p}(\R_{+}^{n})\ \text{ and }\ \hat{g}^{z}\in\bW_{\circ}^{1,p}(\R_{+}^{n}).\label{eq:pf-fg}
\end{align}

To check (\ref{eq:pf-sigma}), we notice that, from Definition \ref{def:domain},
the equation of the surface $B_{\rho_{0}}(z)\cap\pd\Dom$ is $\psi^{1}(x)=0$,
so $\partial\psi^{1}(x)$ is a normal vector of $\partial G$ at $x\in B_{\rho_{0}}(z)\cap\pd\Dom$.
Thanks to Assumption \ref{assu:compatib}, one has that for $x\in B_{\rho_{0}}(z)\cap\pd\Dom$,
\begin{equation}
0=\partial\psi^{1}(x)\cdot\sigma^{\cdot k}(t,x)=\sigma^{rk}(t,x)\partial_{r}\psi^{1}(x)=\tilde{\sigma}^{1k}(t,\psi(x)).\label{eq:pf011}
\end{equation}
Also notice that $\tilde{\sigma}^{1k}(t,\cdot)=0$ outside $\bar{U}_{+}^{z}$.
So (\ref{eq:pf-sigma}) is valid. 

To check (\ref{eq:pf-fg}), one can use Lemmas \ref{lem:equivalent}
and \ref{lem:tilde-a-sigma} to obtain that $\hat{f}^{z}\in\bL^{p}(\R_{+}^{n})$,
$\hat{g}^{z}\in\bW^{1,p}(\R_{+}^{n})$, and
\[
\begin{aligned}\|\hat{f}^{z}\|_{\bL^{p}(\R_{+}^{n})} & \le C\tilde{\cm}(\rho)\|\tilde{u}_{yy}^{z}\|_{\bL^{p}(\R_{+}^{n})}+C\big\{\|\tilde{\zeta}^{z}\tilde{f}\|_{\bL^{p}(\R_{+}^{n})}+\|(\tilde{\zeta}_{yy}^{z}\tilde{u},\tilde{\zeta}_{y}^{z}\tilde{u}_{y},\tilde{\zeta}^{z}\tilde{u}_{y})\|_{\bL^{p}(\R_{+}^{n})}\big\}\\
 & \le C\tilde{\cm}(\rho)\|u_{xx}^{z}\|_{\bL^{p}(G)}+C\big\{\|\zeta^{z}f\|_{\bL^{p}(G)}+\|(\zeta_{xx}^{z}u,\zeta_{x}^{z}u,\zeta^{z}u,\zeta_{x}^{z}u_{x},\zeta^{z}u_{x})\|_{\bL^{p}(G)}\big\}\\
 & \le C\tilde{\cm}(\rho)\|u_{xx}\|_{\bL^{p}(B_{\rho/2}(z)\cap\Dom)}+C\big\{\|f\|_{\bL^{p}(B_{\rho/2}(z)\cap\Dom)}+\rho^{-2}\|u\|_{\bW^{1,p}(B_{\rho/2}(z)\cap\Dom)}\big\},\\
\|\hat{g}^{z}\|_{\bW^{1,p}(\R_{+}^{n};\ell^{2})} & \le C\tilde{L}\rho\|\tilde{u}_{yy}^{z}\|_{\bL^{p}(\R_{+}^{n})}+C\big\{\|\tilde{\zeta}^{z}\tilde{g}\|_{\bW^{1,p}(\R_{+}^{n};\ell^{2})}+\|\tilde{\zeta}_{y}^{z}\tilde{u}\|_{\bW^{1,p}(\R_{+}^{n})}\big\}\\
 & \le C\tilde{L}\rho\|u_{xx}\|_{\bL^{p}(B_{\rho/2}(z)\cap\Dom)}+C\big\{\|g\|_{\bW^{1,p}(B_{\rho/2}(z)\cap\Dom;\ell^{2})}+\rho^{-2}\|u\|_{\bW^{1,p}(B_{\rho/2}(z)\cap\Dom)}\big\}
\end{aligned}
\]
with $C=C(K,n,p,K_{0},\rho_{0},L)$ independent of $\rho$, where
$\tilde{\cm}(\cdot)$ and $\tilde{L}$ are taken from Lemma \ref{lem:tilde-a-sigma}.
It remains to check $\hat{g}^{z}\in\bW_{\circ}^{1,p}(\R_{+}^{n})$.
This immediately follows from some basic facts in real analysis:
\begin{lem}
Let $h$ and $\vf$ be functions defined $\R_{+}^{n}$. Then we have
\begin{enumerate}
\item[\emph{(a)}]  if $h\in W_{\circ}^{1,p}(\R_{+}^{n})$ and $\vf\in C^{0,1}(\bar{\R}_{+}^{n})$,
then $\vf h\in W_{\circ}^{1,p}(\R_{+}^{n})$;
\item[\emph{(b)}] if $h\in W^{1,p}(\R_{+}^{n})$ and $\vf\in C_{0}^{0,1}(\bar{\R}_{+}^{n})$
, then $\vf h\in W_{\circ}^{1,p}(\R_{+}^{n})$;
\item[\emph{(c)}] if $h\in W_{\circ}^{2,p}(\R_{+}^{n})$, then $h_{x^{i}}\in W_{\circ}^{1,p}(\R_{+}^{n})$
for $i=2,\dots,n$,
\end{enumerate}
where $C^{0,1}(\bar{\R}_{+}^{n})$ is the space of all uniformly Lipschitz
continuous functions defined on $\bar{\R}_{+}^{n}$, and its subset
$C_{0}^{0,1}(\bar{\R}_{+}^{n})$ collects those functions that vanish
on the boundary $\{x^{1}=0\}$. 
\end{lem}

Now we use the above lemma to verify $\hat{g}^{z}\in\bW_{\circ}^{1,p}(\R_{+}^{n})$.
By the assertion (a), it is easily seen the last two terms in the
expression (\ref{eq:pr010}) of $\hat{g}^{z}$ belong to $\bW_{\circ}^{1,p}(\R_{+}^{n};\ell^{2})$.
Assumption \ref{assu:continuity} and the condition $\tilde{\sigma}^{1\cdot}(t,0,y')=0$
checked above imply that $\tilde{\sigma}^{1\cdot}(t,\cdot)\in C_{0}^{0,1}(\bar{\R}_{+}^{n};\ell^{2})$
uniformly with respect to $(t,\omega)$, which along with $\tilde{u}_{y^{1}}^{z}\in\bW^{1,p}(\R_{+}^{n})$
yields $[\tilde{\sigma}^{1\cdot}-\tilde{\sigma}^{1\cdot}(\cdot,0)]\tilde{u}_{1}^{z}\in\bW_{\circ}^{1,p}(\R_{+}^{n};\ell^{2})$
by means of the assertion (b). Moreover, because $\tilde{u}^{z}\in\bW_{\circ}^{2,p}(\R_{+}^{n};\ell^{2})$
and $\tilde{\sigma}^{i\cdot}(t,\cdot)\in C^{0,1}(\bar{\R}_{+}^{n};\ell^{2})$,
it follows from the assertion (c) that $[\tilde{\sigma}^{i\cdot}-\tilde{\sigma}^{i\cdot}(\cdot,0)]\tilde{u}_{i}^{z}\in\bW_{\circ}^{1,p}(\R_{+}^{n};\ell^{2})$
for $i=2,\dots,n$. Therefore, we have $\hat{g}^{z}\in\bW_{\circ}^{1,p}(\R_{+}^{n})$.

The facts (\ref{eq:pf-sigma}) and (\ref{eq:pf-fg}) along with Lemma
\ref{lem:tilde-a-sigma} ensure us to apply Proposition \ref{prop:apriori}
to Eq. (\ref{eq:tilde-uz}) to get the estimate
\[
\begin{aligned}\|u\|_{\bW^{2,p}(B_{\rho/4}(z)\cap\Dom)} & \le\|u^{z}\|_{\bW^{2,p}(B_{\rho/2}(z)\cap\Dom)}\le C\|\tilde{u}^{z}\|_{\bW^{2,p}(\R_{+}^{n})}\\
 & \le C\bigl(\|\hat{f}^{z}\|_{\bL^{p}(\R_{+}^{n})}+\|\hat{g}^{z}\|_{\bW^{1,p}(\R_{+}^{n};\ell^{2})}\bigr)\\
 & \le C(\tilde{\cm}(\rho)+\tilde{L}\rho)\|u_{xx}\|_{\bL^{p}(B_{\rho/2}(z)\cap\Dom)}+C\rho^{-2}\|u\|_{\bW^{1,p}(B_{\rho/2}(z)\cap\Dom)}\\
 & \quad+C\big\{\|f\|_{\bL^{p}(B_{\rho/2}(z)\cap\Dom)}+\|g\|_{\bW^{1,p}(B_{\rho/2}(z)\cap\Dom;\ell^{2})}\big\},
\end{aligned}
\]
where $C=C(\kappa,K,n,p,T,K_{0},\rho_{0},L)$. By interpolation, we
have
\begin{align*}
\|u_{x}\|_{\bL^{p}(B_{\rho/2}(z)\cap\Dom)} & \le C(n)\|u_{xx}\|_{\bL^{p}(B_{\rho/2}(z)\cap\Dom)}^{1/2}\|u\|_{\bL^{p}(B_{\rho/2}(z)\cap\Dom)}^{1/2}\\
 & \le\rho^{3}\|u_{xx}\|_{\bL^{p}(B_{\rho/2}(z)\cap\Dom)}^{1/2}+C(n)\rho^{-3}\|u\|_{\bL^{p}(B_{\rho/2}(z)\cap\Dom)}^{1/2}.
\end{align*}
Combining the last two inequalities, we obtain
\begin{equation}
\begin{aligned}\|u\|_{\bW^{2,p}(B_{\rho/4}(z)\cap\Dom)} & \le C(\tilde{\cm}(\rho)+\tilde{L}\rho)\|u_{xx}\|_{\bL^{p}(B_{\rho/2}(z)\cap\Dom)}+C\rho^{-5}\|u\|_{\bL^{p}(B_{\rho/2}(z)\cap\Dom)}\\
 & \quad+C\big\{\|f\|_{\bL^{p}(B_{\rho/2}(z)\cap\Dom)}+\|g\|_{\bW^{1,p}(B_{\rho/2}(z)\cap\Dom;\ell^{2})}\big\}.
\end{aligned}
\label{eq:pf006}
\end{equation}
Now we define the narrow area near the boundary $\partial G$:
\[
\Dom_{r}=\{x\in\Dom:\text{there is an }\bar{x}\in\pd\Dom\text{ such that }|x-\bar{x}|<r\}.
\]
\begin{lem}
\label{lem:cover}There exist countable points $z_{1},z_{2},\dots\in\pd\Dom$
satisfying the following properties:
\begin{enumerate}
\item $|z_{i}-z_{j}|\ge\rho/8$ for $i\neq j$, and the whole $\pd\Dom$
is covered by $\cup_{i}B_{\rho/8}(z_{i})$;
\item any $x\in\Dom_{\rho/8}$ lies in at least one $B_{\rho/4}(z_{i})$;
\item any $x\in G_{\rho/2}$ is covered by at most $N(n)$ balls from $\{B_{\rho/2}(z_{i})\}$,
where $N(n)$ is the greatest number of such points in $B_{1}$ that
any two of them are over $1/4$ apart.
\end{enumerate}
\end{lem}

Now we postpone the proof of this lemma to this end of this subsection
and move on the proof of Proposition \ref{prop:apriori}. From this
lemma it follows that
\begin{align*}
\|u\|_{\bW^{2,p}(\Dom_{\rho/8})} & \le\sum_{i}\|u\|_{\bL^{p}(B_{\rho/2}(z_{i})\cap\Dom)}\le\sum_{i}\|u\|_{\bW^{2,p}(B_{\rho/4}(z_{i})\cap\Dom)}\\
 & \le N(n)\|u\|_{\bW^{2,p}(\Dom_{\rho/2})}\le N(n)\|u\|_{\bW^{2,p}(G)},
\end{align*}
which along with the estimate (\ref{eq:pf006}) yields that
\begin{equation}
\|u\|_{\bW^{2,p}(G_{\rho/8})}\le C(\tilde{\cm}(\rho)+\tilde{L}\rho)\|u_{xx}\|_{\bL^{p}(\Dom)}+C\big\{\rho^{-5}\|u\|_{\bL^{p}(\Dom)}+\|f\|_{\bL^{p}(\Dom)}+\|g\|_{\bW^{1,p}(\Dom;\ell^{2})}\big\}\label{eq:pf008}
\end{equation}
with $C=C(\kappa,K,n,p,T,K_{0},\rho_{0},L)$. 

To obtain the estimate in $G^{\rho/8}:=G\backslash G_{\rho/8}$, we
write $\bar{\zeta}(x)=\zeta(4x)$ and define a cut-off function $\zeta_{0}=\bar{\zeta}*\bm{1}_{G^{\rho/8}}$.
For a solution $u\in\mathcal{W}_{\circ}^{2,p}(\Dom)$ of Eq. (\ref{eq:general-dom}),
the function $u^{0}=\zeta_{0}u\in\mathcal{W}_{\circ}^{2,p}(\R^{n})$,
whose support lies in $\bar{G}$, satisfies the following equation
\[
\md u^{0}=(a^{ij}u_{x^{i}x^{j}}^{0}+f^{0})\vd t+(\sigma^{ik}u_{x^{i}}^{0}+g^{0,k})\vd\BM_{t}^{k},\quad u^{0}(0,\cdot)=0,
\]
on $(0,T]\times\R^{n}$, where
\[
f^{0}=\zeta_{0}f-a^{ij}(\zeta_{0})_{x^{i}x^{j}}u-a^{ij}(\zeta_{0})_{x^{i}}u_{x^{j}},\quad g^{0,k}=\zeta_{0}g^{k}-\sigma^{ik}(\zeta_{0})_{x^{i}}u^{0}.
\]
Thanks to the $L^{p}$-theory of SPDEs in the whole space (cf. Theorem
5.1 in \citet{krylov1999analytic}), we have the estimate
\begin{equation}
\begin{aligned}\|u\|_{\bW^{2,p}(G^{\rho/8})} & \le\|u^{0}\|_{\bW^{2,p}(\R^{n})}\le C\big(\|f^{0}\|_{\bL^{p}(\R^{n})}+\|g^{0}\|_{\bW^{1,p}(\R^{n};\ell^{2})}\big)\\
 & \le C\big(\rho^{-2}\|u\|_{\bW^{1,p}(G)}+\|f^{0}\|_{\bL^{p}(\R^{n})}+\|g^{0}\|_{\bW^{1,p}(\R^{n};\ell^{2})}\big)\\
 & \le C\rho\|u_{xx}\|_{\bL^{p}(G)}+C\big(\rho^{-5}\|u\|_{\bL^{p}(G)}+\|f\|_{\bL^{p}(G)}+\|g\|_{\bW^{1,p}(G;\ell^{2})}\big),
\end{aligned}
\label{eq:pf007}
\end{equation}
where $C=C(\kappa,K,n,p,T,L,\cm)$.

Combining the estimates (\ref{eq:pf008}) and (\ref{eq:pf007}), we
can choose a small number $\rho=\rho(\kappa,K,n,p,T,L,\cm)\in(0,\rho_{0}\wedge1]$
such that
\[
\|u\|_{\bW^{2,p}(G)}\le\frac{1}{2}\|u_{xx}\|_{\bL^{p}(G)}+C\big(\|u\|_{\bL^{p}(G)}+\|f\|_{\bL^{p}(G)}+\|g\|_{\bW^{1,p}(G;\ell^{2})}\big),
\]
which yields
\begin{equation}
\|u\|_{\bW^{2,p}(G)}\le C\big(\|u\|_{\bL^{p}(G)}+\|f\|_{\bL^{p}(G)}+\|g\|_{\bW^{1,p}(G;\ell^{2})}\big).\label{eq:pf009}
\end{equation}

It remains to estimate $\|u\|_{\bL^{p}(G)}$. Applying It\^o's formula
to $\me^{-\lambda t}|u(t,x)|^{p}$ and integrating on $G\times[0,s]\times\PS$,
we have
\begin{equation}
\begin{aligned} & \me^{-\lambda T}\E\|u(T,\cdot)\|_{L^{p}(G)}^{p}+\lambda\E\int_{0}^{T}\me^{-\lambda t}\|u(t,\cdot)\|_{L^{p}(G)}^{p}\vd t\\
 & =p\,\E\int_{0}^{T}\!\!\!\int_{G}\me^{-\lambda t}|u(t,x)|^{p-2}u(t,x)[a^{ij}(t,x)u_{x^{i}x^{j}}(t,x)+f(t,x)]\vd x\md t\\
 & \quad+\frac{1}{2}p(p-1)\E\int_{0}^{T}\!\!\!\int_{G}\me^{-\lambda t}|u(t,x)|^{p-2}\|\sigma^{i\cdot}(t,x)u_{x^{i}}(t,x)+g(t,x)\|_{\ell^{2}}^{2}\vd x\md t\\
 & \le\eps\,\E\int_{0}^{T}\|u_{xx}(t,\cdot)\|_{L^{p}(G)}^{p}\vd t+C(\eps,p,K,T)\E\int_{0}^{T}\me^{-\lambda t}\|u(t,\cdot)\|_{L^{p}(G)}^{p}\vd t\\
 & \quad+C(p,T)\big(\|f\|_{\bL^{p}(G)}^{p}+\|g\|_{\bL^{p}(G;\ell^{2})}^{p}\big).
\end{aligned}
\label{eq:pf016}
\end{equation}
Letting $\lambda=1+C(\eps,p,K,T)$, one can get that
\begin{equation}
\|u\|_{\bL^{p}(G)}^{p}\le\eps C(p,K,T)\|u_{xx}\|_{\bL^{p}(G)}^{p}+C(\eps,p,K,T)\big(\|f\|_{\bL^{p}(G)}^{p}+\|g\|_{\bL^{p}(G;\ell^{2})}^{p}\big).\label{eq:pf017}
\end{equation}
Selecting $\eps>0$ sufficiently small, the above estimate along with
(\ref{eq:pf009}) yields the desired estimate (\ref{eq:apriori}),
so the proof of Proposition \ref{prop:apriori} is complete. 
\end{proof}
With non-homogeneous initial value condition, we have the following
result.
\begin{cor}
\label{cor:linear}Let $\Dom\in\mathcal{C}^{2}$ and Assumptions \ref{assu:compatib}
and \ref{assu:continuity} be satisfied. Suppose that for any $f\in\bL^{p}(G)$
and $g\in\bW_{\circ}^{1,p}(G)$ there exists a unique solution in
$\mathcal{W}_{\circ}^{2,p}(G)$ to Eq. (\ref{eq:general-dom}) with
zero initial-boundary condition. Then for any given $f\in\bL^{p}(G)$,
$g\in\bW_{\circ}^{1,p}(G)$, and
\[
\begin{gathered}u_{0}(\cdot)\in L^{p}(\PS,\Filt_{0},W_{\circ}^{2-2/p,p}(\Dom)),\end{gathered}
\]
Eq. (\ref{eq:general-dom}) with the initial-boundary condition (\ref{eq:bdcondition})
also admits a unique solution $u\in\mathcal{W}_{\circ}^{2,p}(G)$,
and this solution satisfies
\begin{equation}
\|u\|_{\bW^{2,p}(G)}^{p}\le C\bigl(\|f\|_{\bL^{p}(G)}^{p}+\|g\|_{\bW^{1,p}(G;\ell^{2})}^{p}+\E\|u(0,\cdot)\|_{W^{2-2/p,p}(G)}^{p}\bigr),\label{eq:apriori-2}
\end{equation}
where the constant $C$ depends only on $\kappa,K,n,p,T,K_{0},\rho_{0},L,$
and the functions $\cm(\cdot)$.
\end{cor}

\begin{proof}
From Theorem IV.9.1 in \citet{ladyzhenskaya1988linear}, the heat
equation
\[
\partial_{t}V=\lap V\ \text{ on }(0,T]\times G;\quad V(t,\cdot)|_{\partial G}=0;\quad V(0,\cdot)=u(0,\cdot)\ \text{ on }G
\]
has a unique strong solution $V(\cdot,\cdot,\omega)\in L^{p}((0,T),W_{\circ}^{2,p}(G))$
for each $\omega$, and
\begin{equation}
\|V\|\le C(n,p,K_{0},\rho_{0},T)\E\|u(0,\cdot)\|_{W^{2-2/p,p}(G)}^{p}.\label{eq:pf012}
\end{equation}
On the other hand, from the assumptions the following equation
\begin{equation}
\begin{aligned}\md U & =[a^{ij}U_{x^{i}x^{j}}+f+(a^{ij}-\delta^{ij})V_{x^{i}x^{j}}]\vd t+(\sigma^{ik}U_{x^{i}}+g^{k}+\sigma^{ik}V_{x^{i}})\vd\BM_{t}^{k},\\
U|_{\partial G} & =U(0,\cdot)=0
\end{aligned}
\label{eq:pf013}
\end{equation}
has a unique solution $U\in\mathcal{W}_{\circ}^{2,p}(G)$, and from
Proposition \ref{prop:apriori} we have
\begin{align*}
\|U\|_{\bW^{2,p}(G)}^{p} & \le C(\|f+(a^{ij}-\delta^{ij})V_{x^{i}x^{j}}\|_{\bL^{p}(G)}^{p}+\|g+\sigma^{i\cdot}V_{x^{i}}\|_{\bW^{1,p}(G;\ell^{2})}^{p})\\
 & \le C\bigl(\|f\|_{\bL^{p}(G)}^{p}+\|g\|_{\bW^{1,p}(G;\ell^{2})}^{p}+\|V\|_{\bW^{2,p}(G)}^{p}\bigr)\\
 & \le C\bigl(\|f\|_{\bL^{p}(G)}^{p}+\|g\|_{\bW^{1,p}(G;\ell^{2})}^{p}+\E\|u(0,\cdot)\|_{W^{2-2/p,p}(G)}^{p}\bigr).
\end{align*}
Obviously, the function $u=U+V\in\mathcal{W}_{\circ}^{2,p}(G)$ solves
Eq. (\ref{eq:general-dom}) with condition (\ref{eq:bdcondition}),
and (\ref{eq:apriori-2}) immediately follows from the above estimates
for $U$ and $V$. The uniqueness also holds true, otherwise we can
construct different solutions of (\ref{eq:pf013}) from different
solutions of Eq. (\ref{eq:general-dom}) with (\ref{eq:bdcondition})
(with the help of $V$), which contradicts to the assumptions. The
proof is complete.
\end{proof}
\begin{proof}[Proof of Lemma \ref{lem:cover}]
 For convenience, we say $\{z_{1},z_{2},\dots\}$ is a proper $\rho/8$-covering
set of $E\subset\R^{n}$ if $z_{i}\in E$, $|z_{i}-z_{j}|\ge\rho/8$
for $i\neq j$, and $E$ is covered by $\cup_{i}B_{\rho/8}(z_{i})$. 

If $G$ is bounded, then $\partial G$ is a compact subset in $\R^{n}$,
so there are finite points $\{z_{1},\dots,z_{N}\}\subset\partial G$
such that $\partial G\subset\cup_{i}B_{\rho/8}(z_{i})$, but it is
not necessary that $|z_{i}-z_{j}|\ge\rho/8$ for any $i\neq j$. Now
we adjust the choice of points $z_{i}$ as follows: in the $i$-th
step, we check whether $B_{\rho/8}(z_{i})\subset\cup_{j\neq i}B_{\rho/8}(z_{j})$:
if yes, then remove this $z_{i}$ from the set; if no, then $E_{i}:=\partial G\,\backslash\!\cup_{j\neq i}B_{\rho/8}(z_{j})$
is nonempty and covered by $B_{\rho/8}(z_{i})$, so we can pick one
point $z'_{i}\in E_{i}$ or two $z_{i}',z_{i}''\in E_{i}$ with $|z_{i}'-z''_{i}|\ge\rho/8$
such that $E_{i}$ is covered by $B_{\rho/8}(z_{i}')$ or $B_{\rho/8}(z_{i}')\cup B_{\rho/8}(z_{i}'')$,
and replace $z_{i}$ by $z_{i}'$ or the pair $(z_{i}',z_{i}'')$.
After $N$ steps one obtains a finite proper $\rho/8$-covering set
of $\partial G$.

If $G$ is unbounded, we fix a large number $R>0$ and denote $\varGamma_{k}=\partial G\cap B_{kR}(0)$.
Repeating the argument as above one can find a sequence of finite
sets $A_{1},A_{2},\dots$ inductively such that $A_{1}$ is a finite
proper $\rho/8$-covering set of $\varGamma_{1}$, and $A_{k}$ with
$k\ge2$ is a finite proper $\rho/8$-covering set of $\varGamma_{k}\backslash D_{k-1}$,
where $D_{k-1}=\cup\{B_{\rho/8}(z):z\in\cup_{i=1}^{k-1}A_{1}\}$.
It is easily seen that $A:=\cup_{i=1}^{\infty}A_{i}$ is a finite
proper $\rho/8$-covering set of $\partial G$. 

Next we prove that the set $A$ has the second property. For $x\in G_{\rho/8}$
there is an $\bar{x}\in\partial G$ such that $|x-\bar{x}|<\rho/8$.
Meanwhile, there is a point $z\in A$ such that $\bar{x}\in B_{\rho/8}(z)$.
Hence, $|x-z|\le|x-\bar{x}|+|\bar{x}-z|\le\rho/4$, which means $x\in B_{\rho/4}(z)$.
Finally, for $x\in G_{\rho/2}$ the ball $B_{\rho/2}(x)$ contains
at most $N(n)$ points from the set $A$ according to the definition
of $N(n)$, which implies the last property. The proof is complete.
\end{proof}

\subsection{Existence and uniqueness}

We start from the solvability of stochastic heat equations. In view
of Corollary \ref{cor:linear}, we can just focus on the homogeneous
Dirichlet boundary value problem.
\begin{lem}
Let $G\in\mathcal{C}^{2}$. Then for given $f\in\bL^{p}(G)$ and $g\in\bW_{\circ}^{1,p}(G)$,
the equation 
\begin{equation}
\md u=(\lap u+f)\vd t+g^{k}\vd\BM_{t}^{k}\quad\text{on }(0,T]\times G\label{eq:lap-dom}
\end{equation}
with zero initial-boundary condition has a unique solution $u\in\mathcal{W}_{\circ}^{2,p}(G)$.
\end{lem}

\begin{proof}
The uniqueness follows from the estimate (\ref{eq:apriori}). For
the existence we adopt an approximation strategy from the proof of
Theorem 2.9 in \citet{kim2004stochastic}. It is well-known that $C_{0}^{\infty}(G)$
is a dense subset of $W_{\circ}^{1,p}(G)$. We can approximate $g=(g^{1},g^{2},\dots)\in\bW_{\circ}^{1,p}(G;\ell^{2})$
with functions having only finite nonzero entries, bounded on $[0,T]\times G\times\PS$
along with each derivative of any order, and vanishing near $\pd G$
and the infinity (cf. Theorem 3.17 in \citet{adams2003sobolev}).
In this case it is known that
\[
V(t,x)=\int_{0}^{t}g^{k}(t,x)\vd\BM_{s}^{k}
\]
is infinitely differentiable in $x$ and vanishes near $\pd G$ and
the infinity. So we conclude that $V\in\mathcal{W}_{\circ}^{2,p}(G)$.
Again, from PDE theory, the equation
\[
\partial_{t}U=\lap U+f+\lap V,\quad U|_{\pd G}=0,\quad U(0,\cdot)=0
\]
has a solution $U$ in $\mathcal{W}_{\circ}^{2,p}(G)$. The solution
of (\ref{eq:lap-dom}) is then given by $u=U+V\in\mathcal{W}_{\circ}^{2,p}(G)$.
The case of general $g$ can be obtained by approximation by the help
of the estimate (\ref{eq:apriori}). The proof is complete.
\end{proof}
With the solvability of stochastic heat equation (\ref{eq:lap-dom})
and the a priori estimate (\ref{eq:apriori}) in hand, the existence
and uniqueness of solutions to the general linear equation (\ref{eq:general-dom})
immediately follows from the standard method of continuity (cf. \citet[Theorem 5.2]{gilbarg2001elliptic}).
Bearing in mind Corollary \ref{cor:linear}, we have the following
result.
\begin{cor}
\label{cor:nohomog}Let $G\in\mathcal{C}^{2}$ and Assumptions \ref{assu:compatib}
and \ref{assu:continuity}. Then for any given $f\in\bL^{p}(G)$,
$g\in\bW_{\circ}^{1,p}(G)$ and $u_{0}(\cdot)\in L^{p}(\PS,\Filt_{0},W_{\circ}^{2-2/p,p}(\Dom))$,
Eq. (\ref{eq:general-dom}) with the initial-boundary condition (\ref{eq:bdcondition})
has a unique solution $u\in\mathcal{W}_{\circ}^{2,p}(G)$.
\end{cor}

\begin{proof}[Proof of Theorem \ref{thm:main} (i) and (ii)]
 The argument is similar to the proof of Theorem 6.4 in \citet{krylov1999analytic}.
From Assumption \ref{assu:fg}, we know that for any $v\in\bW_{\circ}^{2,p}(G)$,
\[
f(\cdot,\cdot,v)\in\bL^{p}(G),\quad g(\cdot,\cdot,v)\in\bW_{\circ}^{1,p}(G;\ell^{2}).
\]
So by Corollary \ref{cor:nohomog}, the equation
\[
\md u=(a^{ij}u_{x^{i}x^{j}}+f(t,x,v))\vd t+(\sigma^{ik}u_{x^{i}}+g^{k}(t,x,v))\vd\BM_{t}^{k}
\]
with condition (\ref{eq:bdcondition}) has a unique solution $u\in\mathcal{W}_{\circ}^{2,p}(G)$. 

Define a mapping $\mathcal{T}v=u$. Replacing the terminal time $T$
into any $\tau\le T$, it follows from the estimate (\ref{eq:apriori})
and Assumption \ref{assu:fg} that for $v^{1},v^{2}\in\bW_{\circ}^{2,p}(G)$,
\begin{align*}
\|\mathcal{T}v^{1}-\mathcal{T}v^{2}\|_{\bW^{2,p}(G,\tau)}^{p} & \le C(\|f(\cdot,\cdot,v^{1})-f(\cdot,\cdot,v^{2})\|_{\bL^{p}(G,\tau)}^{p}+\|g(\cdot,\cdot,v^{1})-g(\cdot,\cdot,v^{2})\|_{\bW^{1,p}(G,\tau;\ell^{2})}^{p})\\
 & \le C\eps^{p}\|v^{1}-v^{2}\|_{\bW^{2,p}(G,\tau)}^{p}+CK_{\eps}^{p}\int_{0}^{\tau}\|v^{1}(s)-v^{2}(s)\|_{\bL^{p}(G,s)}^{p}\md s.
\end{align*}
From the computation (\ref{eq:pf016}) (with $s$ instead of $T$)
and Assumption \ref{assu:fg}, we can see that
\[
\E\|\mathcal{T}v^{1}(s)-\mathcal{T}v^{2}(s)\|_{L^{p}(G)}^{p}\le C\|v^{1}-v^{2}\|_{\bW^{2,p}(G,s)}^{p}
\]
with $C$ independent of $s$. Combining the last two inequalities
and letting $C\eps^{p}=1/4$ and , we have
\[
\|\mathcal{T}v^{1}-\mathcal{T}v^{2}\|_{\bW^{2,p}(G,\tau)}^{p}\le\frac{1}{4}\|v^{1}-v^{2}\|_{\bW^{2,p}(G,\tau)}^{p}+C\int_{0}^{\tau}\|v^{1}-v^{2}\|_{\bW^{2,p}(G,s)}^{p}\md s,
\]
Then by induction we can compute that for positive integer $m$,
\begin{align*}
\|\mathcal{T}^{m}v^{1}-\mathcal{T}^{m}v^{2}\|_{\bW^{2,p}(G)}^{p} & \le\|v^{1}-v^{2}\|_{\bW^{2,p}(G)}^{p}\sum_{k=0}^{m}\begin{pmatrix}\alpha\\
k
\end{pmatrix}\frac{4^{k-m}}{k!}(CT)^{k}\\
 & \le2^{-m}\|v^{1}-v^{2}\|_{\bW^{2,p}(G)}^{p}\max_{k\ge0}\frac{(4CT)^{k}}{k!}.
\end{align*}
Choose $m$ sufficiently large so that $\mathcal{T}^{m}$ is a contraction
in $\bW_{\circ}^{2,p}(G)$. Then there is a unique $u\in\bW_{\circ}^{2,p}(G)$
such that $\mathcal{T}^{m}u=u$, and from Corollary \ref{cor:nohomog}
we have $u\in\mathcal{W}_{\circ}^{2,p}(G)$. 

Now we derive the estimate (\ref{eq:main-est}). From Corollary \ref{cor:nohomog}
and Assumption \ref{assu:fg} (with a proper choice of $\eps$), we
can obtain that
\[
\|u\|_{\bW^{2,p}(G)}^{p}\le C\bigl(\|u\|_{\bL^{p}(G)}^{p}+\|f(\cdot,\cdot,0)\|_{\bL^{p}(G)}^{p}+\|g(\cdot,\cdot,0)\|_{\bW^{1,p}(G;\ell^{2})}^{p}+\E\|u_{0}\|_{W^{2-2/p,p}(G)}^{p}\bigr).
\]
The term $\|u\|_{\bL^{p}(G)}^{p}$ can be eliminated just as we got
rid of the same one in (\ref{eq:pf009}). The assertions (i) and (ii)
of Theorem \ref{thm:main} are proved.
\end{proof}
In the proof of Theorem \ref{thm:local} we will need the following
result concerning the existence and uniqueness of $W^{1,p}$-solutions
of SPDEs of divergence form. We keep the formulation as the most compact
form that can be applied comfortably, and leave the general extension
to readers.
\begin{prop}
\label{prop:W1p}Let Assumptions \ref{assu:compatib} and \ref{assu:continuity}
be satisfied with $G=\R_{+}^{n}$, and let $c^{i}\in\bL^{\infty}(\R_{+}^{n})$.
Then for any $f^{0},F\in\bL^{p}(\R_{+}^{n})$ and $g\in\bL^{p}(G;\ell^{2})$,
the equation
\begin{align*}
\md u & =[(a^{ij}u_{x^{i}})_{x^{j}}+f^{0}+c^{i}F_{x^{i}}]\vd t+(\sigma^{ik}u_{x^{i}}+g^{k})\vd\BM_{t}^{k},\\
u|_{\partial G} & =0,\quad u|_{t=0}=0
\end{align*}
has a unique solution $u\in\mathcal{W}_{\circ}^{1,p}(\R_{+}^{n})$,
and
\begin{equation}
\|u\|_{\bW^{1,p}(\R_{+}^{n})}\le C(\|(f^{0},F)\|_{\bL^{p}(\R_{+}^{n})}+\|g\|_{\bL^{p}(\R_{+}^{n};\ell^{2})});\label{eq:divergence}
\end{equation}
where the constant $C$ depends only on $\kappa,K,n,p,T,L,\cm(\cdot)$
and $\|c^{i}\|_{\bL^{\infty}}$.
\end{prop}

\begin{proof}
As above the existence and uniqueness of solutions follows from the
(a priori) estimate (\ref{eq:divergence}) by using the method of
continuity and the Banach fixed-point theorem. The proof of (\ref{eq:divergence})
is similar to but much easier than the derivation of estimate (\ref{eq:apriori})
because one needn't straighten the boundary but just do the computation
on the original equation, while the auxiliary estimate for model equations
is provided by Proposition \ref{prop:apriori} (ii). We suppress the
details here to avoid unnecessary repeating.
\end{proof}

\subsection{Embedding for $\mathcal{W}_{\circ}^{2,p}(G)$}

Let us define the following norm for the space $\mathcal{W}_{\circ}^{2,p}(G)$
(recall Definition \ref{def:W2p}):
\[
\|u\|_{\mathcal{W}_{\circ}^{2,p}(G)}=\|u_{xx}\|_{\bL^{p}(G)}+\|u_{{\rm D}}\|_{\bL^{p}(G)}+\|u_{{\rm S}}\|_{\bW^{1,p}(G;\ell^{2})}+(\E\|u_{0}\|_{W^{2-2/p,p}(G)}^{p})^{1/p}.
\]
Following the proof of Theorem 3.7 in \citet{krylov1999analytic},
one can prove that $\mathcal{W}_{\circ}^{2,p}(G)$ is a Banach space
with the above norm. 

The assertion (iii) of Theorem \ref{thm:main} is a direct consequence
of the following lemma.
\begin{lem}
\label{lem:embedding}Let $G\in\mathcal{C}^{2}$ and $p>2$. Then
for $u\in\mathcal{W}_{\circ}^{2,p}(G)$ we have

(a) if $\alpha_{0}:=\frac{2p-n-2}{2p}>0$, then for any $\alpha\in(0,\alpha_{0})$,
\[
\E\|u\|_{C^{\alpha/2,\alpha}([0,T]\times\bar{G})}^{p}\le C(n,p,\alpha,K_{0},\rho_{0},T)\|u\|_{\mathcal{W}_{\circ}^{2,p}(G)}^{p};
\]

(b) if $\beta_{0}:=\frac{p-n-2}{2p}>0$, then for any $\beta\in(0,\beta_{0})$,
\[
\E\|u_{x}\|_{C^{\beta/2,\beta}([0,T]\times\bar{G})}^{p}\le C(n,p,\beta,K_{0},\rho_{0},T)\|u\|_{\mathcal{W}_{\circ}^{2,p}(G)}^{p}.
\]
\end{lem}

\begin{proof}
When $G=\R^{n}$ this lemma is a simple consequence of Theorem 7.2
in \citet{krylov1999analytic} by means of Sobolev embedding. 

For $G=\R_{+}^{n}$ it suffices to show that the odd extension of
$u$ (see (\ref{eq:odd-ext})) lies in $\mathcal{W}_{\circ}^{2,p}(\R^{n})$.
Indeed, we set $f=u_{{\rm D}}-\lap u\in\bL^{p}(\R_{+}^{n})$ and $g=u_{{\rm S}}\in\bW_{\circ}^{1,p}(\R_{+}^{n};\ell^{2})$,
then
\begin{equation}
\md u=(\lap u+f)\vd t+g^{k}\vd\BM_{t}^{k}.\label{eq:pf019}
\end{equation}
We continue $u_{0}$, $f$ and $g$ to be odd functions of $x^{1}$,
and solve the above equation with initial data $u_{0}$ in the whole
space $\R^{n}$. By our solvability results, the solution of the extended
equation is the odd continuation of $u$ and belongs to $\mathcal{W}_{\circ}^{2,p}(\R^{n})$.

Finally, we consider the case of general $G\in\mathcal{C}^{2}$. For
$u\in\mathcal{W}_{\circ}^{2,p}(G)$, and define $\tilde{u}^{z}$,
$\tilde{u}_{{\rm D}}^{z}$ and $\tilde{u}_{{\rm S}}^{z}$ in the spirit
of (\ref{eq:pr018}) for any $z\in\partial G$. Evidently, $\tilde{u}^{z}\in\mathcal{W}_{\circ}^{2,p}(\R_{+}^{n})$
with $\md\tilde{u}^{z}=\tilde{u}_{{\rm D}}^{z}\vd t+\tilde{u}_{{\rm S}}^{z,k}\vd\BM_{t}^{k}$.
Bearing in mind the assertion for $\R_{+}^{n}$, a direct computation
shows that
\begin{align*}
 & \E\|u\|_{C^{\alpha/2,\alpha}([0,T]\times(\bar{B}_{\rho_{0}/4}(z)\cap\bar{G}))}^{p}\le\E\|\zeta^{z}u\|_{C^{\alpha/2,\alpha}([0,T]\times(\bar{B}_{\rho_{0}/2}(z)\cap\bar{G}))}^{p}\\
 & \le C(n,p,\alpha,K_{0},\rho_{0})\E\|\tilde{u}^{z}\|_{C^{\alpha/2,\alpha}([0,T]\times\bar{\R}_{+}^{n})}^{p}\le C(n,p,\alpha,K_{0},\rho_{0},T)\|\tilde{u}^{z}\|_{\mathcal{W}_{\circ}^{2,p}(\R_{+}^{n})}^{p}\\
 & \le C(n,p,\alpha,K_{0},\rho_{0},T)\|u\|_{\mathcal{W}_{\circ}^{2,p}(G)}^{p},\\
 & \E\|u_{x}\|_{C^{\beta/2,\beta}([0,T]\times(\bar{B}_{\rho_{0}/4}(z)\cap\bar{G}))}^{p}\le\E\|(\zeta^{z}u)_{x}\|_{C^{\beta/2,\beta}([0,T]\times(\bar{B}_{\rho_{0}/2}(z)\cap\bar{G}))}^{p}\\
 & \le C(n,p,\beta,K_{0},\rho_{0})\E\|\partial\tilde{u}^{z}\|_{C^{\beta/2,\beta}([0,T]\times\bar{\R}_{+}^{n})}^{p}\le C(n,p,\beta,K_{0},\rho_{0})\|\tilde{u}^{z}\|_{\mathcal{W}_{\circ}^{2,p}(\R_{+}^{n})}^{p}\\
 & \le C(n,p,\beta,K_{0},\rho_{0})\|u\|_{\mathcal{W}_{\circ}^{2,p}(G)}^{p}.
\end{align*}
For $z\in G^{\rho_{0}/4}=G\backslash G_{\rho_{0}/4}$ the estimate
is much simpler:
\begin{align*}
\E\|u\|_{C^{\alpha/2,\alpha}([0,T]\times\bar{B}_{\rho_{0}/8}(z))}^{p} & \le\E\|\eta^{z}u\|_{C^{\alpha/2,\alpha}([0,T]\times\R^{n})}^{p}\le C\|\eta^{z}u\|_{\mathcal{W}_{\circ}^{2,p}(\R^{n})}^{p}\le C\|u\|_{\mathcal{W}_{\circ}^{2,p}(G)}^{p},\\
\E\|u_{x}\|_{C^{\beta/2,\beta}([0,T]\times\bar{B}_{\rho_{0}/8}(z))}^{p} & \le\E\|(\eta^{z}u)_{x}\|_{C^{\beta/2,\beta}([0,T]\times\R^{n})}^{p}\le C\|\eta^{z}u\|_{\mathcal{W}_{\circ}^{2,p}(\R^{n})}^{p}\le C\|u\|_{\mathcal{W}_{\circ}^{2,p}(G)}^{p},
\end{align*}
where $\eta^{z}\in C_{0}^{\infty}(\R^{n})$ such that $\eta^{z}(x)=1$
for $|x-z|\le\rho_{0}/8$ and $\eta^{z}(x)=0$ for $|x-z|\ge\rho_{0}/4$.
Therefore, we have bounded the H\"oler norms in any $\bar{B}_{\rho_{0}/8}(z)\cap\bar{G}$
with $z\in\bar{G}$. The desired global estimate follows from the
localization property of H\"oler norms (cf. Theorem 4.1.1 in \citet{krylov1996lectures}).
The lemma is proved. 
\end{proof}

\section{Proof of Theorem \ref{thm:local}}

The interior regularity of the solution is implied in the assumption
$u\in\bW_{{\rm loc}}^{2,p}(G)$. To prove the regularity near $\Gamma':=\Gamma\cap\partial G'$,
it suffices to do this in a neighbourhood of any point $z\in\Gamma'$
because $G'$ is bounded (and $\Gamma'$ is bounded too). In other
words, we need prove that $u\in\bW^{2,p}(B_{\eps}(z)\cap G)$, where
$\eps>0$ is a number much smaller than $\mathrm{dist}(G',\partial G\backslash\Gamma)$
and $\rho_{0}$ (recall Definition \ref{def:domain}). In the spirit
of the method of straightening boundary as in the proof of Proposition
\ref{prop:apriori}, the desired result can be converted equivalently
to the following lemma.
\begin{lem}
The conclusion of Theorem \ref{thm:local} holds true for $G=\R_{+}^{n}$,
$\Gamma=\partial\R_{+}^{n}\cap B_{2\eps}(0)$ and $G'=B_{\eps}(0)$. 
\end{lem}

\begin{proof}
In view of Corollary \ref{cor:linear} one can assume that $u(0,\cdot)=0$.
Take a function $\zeta\in C_{0}^{\infty}(\R^{n})$ such that $\zeta(x)=1$
for $|x|\le3\eps/2$ and $\zeta(x)=0$ for $|x|\ge2\eps$. Then $v=\zeta u$
satisfies the following equation
\[
\md v=[(a^{ij}v_{x^{i}})_{x^{j}}+f^{0}+c^{i}F_{x^{i}}]\vd t+(\sigma^{ik}v_{x^{i}}+\tilde{g}^{k})\vd\BM_{t}^{k},
\]
where
\begin{align*}
c^{i} & =b^{i}\zeta-2a^{ij}\zeta_{x^{j}}-a_{x^{j}}^{ij}\zeta,\quad F=u\\
f^{0} & =\zeta f+(c\zeta-a^{ij}\zeta_{x^{i}x^{j}}-a_{x^{j}}^{ij}\zeta_{x^{i}})u,\\
\tilde{g}^{k} & =\zeta g^{k}+(\nu^{k}\zeta-\sigma^{ik}\zeta_{x^{i}})u.
\end{align*}
From Proposition \ref{prop:W1p} one has $\zeta u=v\in\mathcal{W}_{\circ}^{1,p}(\R_{+}^{n})$
and 
\[
\|u\|_{\bW^{1,p}(B_{3\eps/2})}\le C(\|u\|_{\bL^{p}(\R_{+}^{n})}+\|f\|_{\bL^{p}(\R_{+}^{n})}+\|g\|_{\bL^{p}(\R_{+}^{n};\ell^{2})}).
\]
Now we let $\tilde{\zeta}\in C_{0}^{\infty}(\R^{n})$ such that $\tilde{\zeta}(x)=1$
for $|x|\le\eps$ and $\tilde{\zeta}(x)=0$ for $|x|\ge3\eps/2$.
Then $\tilde{v}=\tilde{\zeta}u$ satisfies
\[
\md\tilde{v}=(a^{ij}\tilde{v}+\tilde{f})\vd t+(\sigma^{ik}\tilde{v}+\tilde{g}^{k})\vd\BM_{t}^{k},
\]
where $\tilde{g}$ is defined above and
\[
\tilde{f}=\tilde{\zeta}f+(b\tilde{\zeta}-a^{ij}\tilde{\zeta}_{x^{j}})u_{x}+(c\tilde{\zeta}-a^{ij}\tilde{\zeta}_{x^{i}x^{j}})u.
\]
Since $u\in\bW^{1,p}(B_{3\eps/2}\cap\R_{+}^{n})$ and $u=0$ on $B_{3\eps/2}\cap\partial\R_{+}^{n}$,
one has $\tilde{f}\in\bL^{p}(\R_{+}^{n})$ and $\tilde{g}\in\bW_{\circ}^{1,p}(\R_{+}^{n})$.
Then from Theorem \ref{thm:main} one obtains $\tilde{\zeta}u=\tilde{v}\in\mathcal{W}_{\circ}^{2,p}(\R_{+}^{n})$.
The continuity property of $\tilde{\zeta}u$ and its derivatives follows
from Lemma \ref{lem:embedding}. The proof is complete.
\end{proof}

\section*{References}

\bibliographystyle{authordate3}
\bibliography{ref}

\end{document}